\documentclass[12pt]{amsart}
\usepackage{amssymb}
\usepackage{amssymb,amscd,amsmath,amsthm}

\def\QQ{{\mathbb Q}}
\def\ZZ{{\mathbb Z}}
\def\RR{{\mathbb R}}
\def\CC{{\mathbb C}}
\def\BB{{\mathbf B}}

\newtheorem{thm}{Theorem}

\newtheorem{prop}[thm]{Proposition}
\newtheorem{lemma}[thm]{Lemma}
\newtheorem{remark}[thm]{Remark}
\newtheorem{example}[thm]{Example}
\newtheorem{definition}[thm]{Definition}

\begin{document}

\title[Homogeneous Fourier Matrices]{Classification of Homogeneous Fourier Matrices}

\author[Gurmail Singh]{Gurmail Singh}
\address{Department of Mathematics and Statistics, University of Regina, Regina, Canada, S4S 0A2}
\email{singh28g@uregina.ca}

\date{}



\begin{abstract} Modular data are commonly studied in mathematics and physics.
A modular datum defines a finite-dimensional representation of the modular group $SL_2(\ZZ)$. In this paper, we show that there is a one-to-one correspondence between  Fourier matrices associated to modular data and self-dual $C$-algebras that satisfy a certain condition. Also, we prove that a homogenous $C$-algebra arising from a Fourier matrix has all the degrees equal to $1$.
\end{abstract}

\maketitle

\section{Introduction}

The set of columns of a Fourier matrix with entrywise multiplication  and usual addition  generate a fusion algebra over $\CC$, see \cite{MC1}.
But we apply a two-step rescaling on Fourier matrices to obtain the $C$-algebras that have nicer properties than the fusion algebras. To establish the results of this paper we use the   properties of Fourier matrices and associated $C$-algebras.

\smallskip

Let $A$ be a finite dimensional and commutative algebra over $\CC$ with  distinguished basis  $\mathbf{B} = \{b_0:=1_A, b_1, \hdots, b_{r-1}\}$, and an $\RR$-linear and $\CC$-conjugate linear involution $*:A \rightarrow A$. Let $\delta: A \rightarrow \CC$ be an algebra homomorphism. Then the triple $(A,\BB,\delta)$ is called a \emph{$C$-algebra} if  it satisfies the following properties:

\begin{enumerate}

\item for all $b_i \in \mathbf{B}$, $(b_i)^* = b_{i^*} \in \mathbf{B}$,
\item 
    for all $b_i, b_j \in \mathbf{B}$, we have
$b_ib_j = \sum_{b_k \in \mathbf{B}}^{}\lambda_{ijk}b_k, \mbox{ for some } \lambda_{ijk} \in \mathbb{R},$
\item  for all $b_i, b_j \in \mathbf{B}, \lambda_{ij0} \neq 0 \iff j = i^*$,
\item for all $b_i \in \mathbf{B}, \lambda_{ii^*0} = \lambda_{i^*i0} > 0$.
\item for all $b_i\in \BB$, $\delta(b_i)=\delta(b_{i^*})>0$.
\end{enumerate}

The algebra homomorphism $\delta$ is called a \emph{degree map}, and $\delta(b_i)$, for all $b_i\in \BB$, are  called the degrees.
For $i\neq 0$, $\delta(b_i)$ is called a \emph{nontrivial degree}. 
If  $\delta(b_i)=\lambda_{ii^*0}$, for all $b_i\in \BB$, we say  that $\BB$ is a  standard basis. The order of a $C$-algebra is denfined as  $n := \delta(\mathbf{B}^+) = \sum_{i=0}^{r-1} \delta(b_i)$.
A $C$-algebra is called \emph{symmetric} if $b_{i^*}=b_i$, for all $i$.
A $C$-algebra with rational structure constants is called a \emph{rational $C$-algebra}.  
The readers interested in $C$-algebras are directed to \cite{AFM}, \cite{HIB1}, \cite{HIB2},  \cite{HS2} and \cite{Hig87}.
%

%
%

\smallskip

In Section 2, we construct the $C$-algebras from Fourier matrices. 
In Section 3,  we establish some  properties of $C$-algebras that arise from Fourier matrices. Also, we prove that there is a one-to-one correspondence between Fourier matrices and  self-dual $C$-algebras that satisfy a certain integrality condition.
In Section 4, we prove that every  homogeneous $C$-algebra arising from a Fourier matrix has all the degrees equal to $1$, and such a $C$-algebra is a group algebra under a certain condition.

\section{$C$-algebras arising from Fourier matrices}

In this section, we introduce the  two-step rescaling on Fourier matrices associated to modular data and the construct $C$-algebras.
To keep the generality, in the following definition of Modular data we assume the structure constants to be integers instead of nonnegative integers, see  \cite{MC1}.

\begin{definition}\label{ModularDef} Let $r\in \ZZ^+$ and  $I$ an $r \times r$ identity matrix. A pair $(S,T)$ of $r\times r$ complex matrices is called modular datum if
\begin{enumerate}
\item $S$ is a unitary and symmetric matrix, that is, $S\bar{S}^T = 1, S= S^T$,
\item  $T$ is diagonal matrix and of finite multiplicative order,
\item $S_{i0}>0$, for $0\leq i\leq r-1$, where $S$ is indexed by $\{0,1,2,\dots,r-1\}$,
\item $(ST)^3=S^2$,
\item $N_{ijk}= \sum_l{S_{li}S_{lj}\bar S_{lk}}{S^{-1}_{l0}} \in \ZZ$, for all $0\leq i,j,k\leq r-1$.
\end{enumerate}
 \end{definition}

\begin{definition}
A matrix $S$ satisfying the axioms $(i), (iii)$ and $(v)$ of Definition \ref{ModularDef} is called a \emph{Fourier matrix}.
\end{definition}

Let $S$ be a Fourier matrix. Let $s=[s_{ij}]$ be the  matrix with entries $s_{ij}=S_{ij}/S_{i0}$, for all $i,j$, and we call it an \emph{$s$-matrix associated to $S$} (briefly, $s$-matrix). Since $S$ is a unitary matrix, $s\bar s^T$=diag$(d_0,d_1,\dots, d_{r-1})$ is a diagonal matrix, where $d_i=\sum_j s_{ij}\bar s_{ij}$. The numbers $d_i$ are called norms of $s$-matrix. The \emph{principal norm} $d_0~ (= S^{-2}_{00})$ is also known as the \emph{size of the modular datum}, see \cite[Definition 3.8]{MC1}.  The relation $s_{ij}=S_{ij}/S_{i0}$ implies the structure constants $N_{ijk}=\sum_ls_{li}s_{lj}s_{lk}d^{-1}_l$, for all $i,j,k$. Since the  structure constants $N_{ijk}$ generated by the columns of $S$ under entrywise multiplication are integers, the numbers $S_{ij}/S_{i0}$ are algebraic integers, see \cite[Section 3]{MC1}. Therefore, if $S$ has only rational entries then the entries of $s$-matrix are rational integers, and such $s$-matrices are known as \emph{integral Fourier matrices}, see \cite[Definition 3.1]{MC1}. Cuntz studied the integral Fourier matrices, see \cite{MC1}.
In this paper,
we consider the broader class of $s$-matrices with algebraic integer entries.
Note that an $s$-matrix satisfies all the axioms of the definition of an integral Fourier matrix except the entries may not be integers. For example, the first eigenmatrix (character table) of every group algebra of a group of prime order is an $s$-matrix but not an integral Fourier matrix.

\smallskip

There is an interesting row-and-column operation (\emph{two-step rescaling}) procedure that can be applied to a Fourier matrix $S$ that results in the first eigenmatrix, the character table,  of a
$C$-algebra, see Theorem \ref{C-Algebra}.
The steps of the procedure are reversed to obtain the Fourier matrix $S$ from  the first eigenmatrix. The explanation of the procedure is as follows. Let $S=[S_{ij}]$ be a Fourier matrix indexed with $\{0,1,\dots, r-1\}$. We divide each row of $S$ with its first entry and obtain the $s$-matrix. The multiplication of each column of the $s$-matrix with its first entry gives the \emph{$P$-matrix associated to $S$} (briefly, $P$-matrix),  the first eigenmatrix of the $C$-algebra. That is, $s_{ij}=S_{ij}S^{-1}_{i0}$ and $p_{ij}=s_{ij}s_{0j}$, for all $i,j$, where $p_{ij}$ denotes the $(i,j)$-entry of the $P$-matrix. Conversely, to obtain the $s$-matrix from a $P$-matrix, divide each column of the $P$-matrix with the squareroot of its first entry.
Further, the Fourier matrix $S$ is obtained from the $s$-matrix by dividing the  $i$th row of $s$-matrix by $\sqrt{d_i}$, where $d_i=\sum_j|s_{ij}|^2$.
That is, $s_{ij}=p_{ij}/\sqrt{p_{0j}}$, and $S_{ij}=s_{ij}/\sqrt{d_i}$,  $\forall ~i,j$. Since the entries of an $s$-matrix are algebraic integers, the entries of $P$-matrix are also algebraic integers.

\begin{remark}\rm Throughout this paper, unless mentioned explicitly, the set of the columns of $P$-matrix and $s$-matrix are denoted by $\BB=\{b_0,b_1,\hdots, b_{r-1}\}$ and $\tilde \BB=\{\tilde b_0, \tilde b_1,\hdots, \tilde b_{r-1}\}$, respectively. The structure constants generated by the columns, with the entrywise multiplication, of $P$-matrix and $s$-matrix are denoted by $\lambda_{ijk}$ and $N_{ijk}$, respectively. $M^T$ denotes the transpose of a matrix $M$.
\end{remark}

The columns of a Fourier matrix gives rise to a fusion algebra, see \cite{MC1}. In the next theorem we show that corresponding to every Fourier matrix there is a $C$-algebra.

\begin{thm}\label{C-Algebra}Let $S$ be a Fourier matrix.
Then the vector space $A:=\CC\BB$ is a $C$-algebra of order $d_0$, $\BB$ is the standard basis of $A$, and $P$-matrix is the first eigenmatrix of $A$.
\end{thm}

\begin{proof}
The $\CC$-conjugate linear involution $*$ on columns of $S$ is given by the involution on elements of $S$,  defined as
$(S_{ij})^*=S_{ij^*}=\bar S_{ij}$ for all $i,j$.
Therefore, if $S_j$ denotes the $j$th column of a Fourier matrix $S$ then the involution on $S_j$ is given by $(S_j)^*=S_{j^*}=[S_{0j^*}, S_{1j^*}, \hdots, S_{(r-1),j^*}]^T$. Since $b_i=s_{0i}\tilde b_{i}$, the structure constants generated by the basis $\BB$ are given by $\lambda_{ijk}={N_{ijk}s_{0i}s_{0j}}{s^{-1}_{0k}}$, for all $i,j, k$.  As $S$ is a unitary matrix, therefore, $N_{ij0}= \sum_l{S_{li}S_{lj}\bar S_{l0}}{S^{-1}_{l0}}\neq 0 \iff j=i^*$ and $N_{ii^*0}=1>0$, for all $i,j$. Hence, $\lambda_{ij0}\neq 0 \iff j=i^*$  and $\lambda_{ii^*0}>0$, for all $i,j$.

Define a $\CC$-conjugate linear map $\delta:A\longrightarrow \CC$  as  $\delta(\sum_ia_i\tilde{b_i})=\sum_i\bar a_is_{0i}$. Thus $\delta(b_i)= \delta(s_{0i}\tilde b_{i}) = s^2_{0i}$, hence $\delta$ is positive valued. The map $\delta$ is an algebra homomorphism, to see:
$$\begin{array}{rcl}\delta(b_i b_j)&=&
s_{0i}s_{0j}\sum\limits^{}_{k}N_{ijk}\delta(\tilde{b_k})
\\&=&s_{0i}s_{0j}[\dfrac{s_{0i}s_{0j}}{d_0} \sum\limits^{r-1}_{k=0}\bar s_{0k}s_{0k}]+s_{0i}s_{0j}[\sum\limits^{r-1}_{l=1}\dfrac{s_{li}s_{lj}}{d_l}
\sum\limits^{r-1}_{k=0}\bar s_{lk}s_{0k}]
\\&=&s_{0i}s_{0j}[\dfrac{s_{0i}s_{0j}}{d_0} d_0]+s_{0i}s_{0j}[0]
\\&=&s^2_{0i}s^2_{0j}=\delta(b_i)\delta(b_j)\end{array}$$
In the third last equality we use the fact that $ \sum\limits^{n}_{k=1} s_{0k}\bar  s_{0k} =d_0$ and the rows of $s$-matrix are orthogonal. Since $b_{i^*}= \bar s_{0i}\tilde{b}_{i^*}=s_{0i}\tilde{b}_{i^*}$, $(b_ib_{i^*})_0=s_{0i}s_{0i}(\tilde b_i \tilde b_{i^*})_0= s^2_{0i}=\delta(b_i)$. Therefore, $\delta$ is the positive degree map and $\BB$ is the standard basis. The value of a basis element $b_j$ under an irreducible character (linear character) is given by $p_{ij}$, for all $i,j$.
\end{proof}

By Theorem \ref{C-Algebra}, every Fourier matrix gives rise to a $C$-algebra.
The algebra $A$ in the above theorem is denoted by $(A, \BB, \delta)$, and we say  \emph{$(A,\BB, \delta)$ is arising from a Fourier matrix $S$}. A $C$-algebra arising from a Fourier matrix $S$ of rank $r$ is a symmetric $C$-algebra if and only if  $S \in \RR^{r \times r}$.

\section{General results on $C$-algebras arising from Fourier matrices}

Let $s$ be an integral Fourier matrix.  Then $\sqrt{d_j}s_{ij}=\sqrt{d_i}s_{ji}$, for all $i,j$.
Therefore,  $d_0=d_i\delta(b_i)$ for all $i$, that is, the degrees and norms divide $d_0$. This can be generalized to $s$-matrices under a certain condition.
In the following proposition we prove that not only the list of the degrees of a $C$-algebra arising from a Fourier matrix matches with list of multiplicities but their indices also match. Also, we prove that a $C$-algebra arising from an integral Fourier matrix has perfect square integral degrees and rational structure constants.

\begin{prop}\label{SymmetrizingPropertyProp}
Let $(A,\BB, \delta)$ be a $C$-algebra arising from a Fourier matrix $S$. Let $m_j$ be the multiplicity of $A$ corresponding to the irreducible character $\chi_j$.
\begin{enumerate}
\item The degrees of $A$ exactly match with the multiplicities of $A$, that is, $m_j=\delta(b_j)$, for all $j$.
\item If degrees of $A$ are rational numbers then the degrees and norms are integers, and both the degrees and norms divide the order of $A$.
\item If the associated $s$-matrix has integral entries then  degrees of $A$ are perfect square integers, and $A$ is a rational $C$-algebra.
\item If $A$ is a rational $C$-algebra the degrees of $A$ are integers.
\item $A$ has unique degree if and only if $s$-matrix has unique norm.
\end{enumerate}
\end{prop}

\begin{proof}
$(i)$. The multiplicities, $m_j = \delta(\BB^+)/\sum\limits_{i=0}^{r-1}\frac{|\chi_j(b_{i^*})|^2}{\lambda_{ii^*0}}$, see \cite[Corollary 5.6]{BI}. Therefore,
$m_j=d_0/\sum\limits_{i=0}^{r-1}|s_{ji}|^2
=d_0/d_j$, for all $j$.
Since $S$ is a symmetric matrix, $d_i|s_{ji}|^2 = d_j|s_{ij}|^2$. Thus  $d_0/d_j = s^2_{0j}$, for all $j$. Hence $m_j=d_0/d_j=s^2_{0j}=\delta(b_j)$, for all $j$.

\smallskip

(ii). The fact that the $P$-matrix has algebraic entries implies that the degrees of the algebra  $A$ are integers, consequently $d_0$ is an integer. By the above part (i), $d_j=d_0\delta(b_j)^{-1}$, for all $j$. Therefore,  $d_j$ are rational numbers. The entries of $s$-matrix are algebraic integers, thus the norms $d_j$ are rational integers. Hence both the degrees and norms divide the order of $A$, because $d_0=d_j\delta(b_j)$, for all $j$.

\smallskip

(iii). The entries of $s$-matrix are integers and $\sqrt{\delta(b_j)}=s_{0j}$ imply $\sqrt{\delta(b_j)}\in \ZZ$, for all $j$.
By the proof of Theorem \ref{C-Algebra}, $\lambda_{ijk}= N_{ijk}
\sqrt{\delta(b_i)}\sqrt{\delta(b_j)}/\sqrt{\delta(b_k)}$, for all $i,j,k$. Hence $\lambda_{ijk} \in \QQ$, for all $i,j,k$.

\smallskip

(iv). Since $A$ is a rational algebra, $\lambda_{ijk}\in \QQ$. Therefore, $\delta(b_i)=\lambda_{ii^*0} \in \QQ$, because $\BB$ is a standard basis.  But the  entries of $P$-matrix are algebraic integers and $\delta(b_i)=p_{0i}$, for all $i$. Therefore,  $\delta(b_i)$ are integers. (Though we remark that rationality of all the structure constants is not required.)

\smallskip

(v). The result follows from the fact that  $d_0=d_i\delta(b_i)$, for all $i$.
\end{proof}

By Proposition \ref{SymmetrizingPropertyProp} (i), a $C$-algebra arising from a Fourier matrix $S$ is a self-dual $C$-algebra. But every self-dual $C$-algebra not necessarily  arise from a Fourier matrix, see Example \ref{AllenIntConditionNotSatisfiedEx}. Therefore, in general, the converse is not true. In the next theorem we prove that the converse is also true under a certain integrality condition.

\begin{thm}\label{IntegralCAlgebrasAndFourierMatrices}
Let $(A,\BB, \delta)$ be a $C$-algebra with standard basis $\BB=\{b_0,b_1, \hdots, b_{r-1}\}$. Let $\lambda_{ijk}$ be the structure constants generated by the basis $\BB$. Then $A$ is self-dual and ${\lambda_{ijk}\sqrt{\delta(b_k)}}/{\sqrt{\delta(b_i)\delta(b_j)}}\in \ZZ$ if and only if $A$ arises from a Fourier matrix $S$.
\end{thm}

\begin{proof}
Suppose
$(A,\BB, \delta)$ is a self-dual $C$-algebra and ${\lambda_{ijk}\sqrt{\delta(b_k)}}/{\sqrt{\delta(b_i)\delta(b_j)}}\in \ZZ$, for all $i,j,k$.
Let $P$ be the first eigenmatrix of $A$ and $I$ an identity matrix.
Therefore, without loss of generality, let $P\bar P=d_0I$. Thus by \cite[Theorem 5.5 (i)]{BI}, $p_{ji}/\delta(b_i)=p_{ij}/m_j$, for all $i,j$, where $m_j$ is the multiplicity of an irreducible character $\chi_j$. By Proposition \ref{SymmetrizingPropertyProp}, $m_j=\delta(b_j)$ for all $j$, therefore,  $p_{ji}/\delta(b_i)=p_{ij}/\delta(b_j)$, for all $i,j$.
Let $L$=diag$(1/\sqrt{\delta(b_0)},1/\sqrt{\delta(b_1)}, $ $\hdots, 1/\sqrt{\delta(b_{r-1}})$, and $s=PL$. Therefore, $s\bar s^T$=diag$(d_0/\delta(b_0),d_0/\delta(b_1), \hdots, d_0/\delta(b_{r-1}))$= diag$(d_0,d_1, \hdots, d_{r-1})$, where $d_0=d_i\delta(b_i)$, $d_i=\sum_j|s_{ij}|^2$ and $s_{ij}=p_{ij}/\sqrt{\delta(b_j)}$, for all $i, j$.
Also, $s_{ij}\sqrt{d_j}
=s_{ij}\sqrt{d_i}$, because $p_{ji}/\delta(b_i)=p_{ij}/\delta(b_j)$, for all $i,j$.. Hence $s$ is an $s$-matrix associated to a Fourier matrix $S=[S_{ij}]$, where $S_{ij}=s_{ij}/\sqrt{d_i}$, for all $i,j$. Since $N_{ijk} \in \ZZ$, the other direction follows from the Proposition \ref{SymmetrizingPropertyProp}(i).
\end{proof}

\begin{remark}\label{GroupAlgIsArsingAllenMatrix}\rm
By Pontryagin duality,  every group algebra of a finite abelian group is self-dual. A group algebra of a finite abelian group also satisfies the integrality condition of the above theorem, because it has all degrees (augmentations) equal to $1$.
Therefore, a group algebra of a finite abelian group is arising from a Fourier matrix.
On the other hand, a self-dual $C$-algebra with a unique degree and nonnegative structure constants is a group algebra with basis a finite abelian group, see Theorem  \ref{ClassificationOfAllenMatricesThm}. Hence a self-dual $C$-algebra with nonnegative structure constants has a unique degree if and only if it is a group algebra of a finite abelian group.
\end{remark}

In the next example, we show that the condition ${\lambda_{ijk}\sqrt{\delta(b_k)}}/{\sqrt{\delta(b_i)\delta(b_j)}}\in \ZZ$ cannot be removed.

\begin{example}\label{AllenIntConditionNotSatisfiedEx}\rm
Since the row sum of a character table is zero, the character table of a $C$-algebra of rank $2$ with basis $\BB=\{b_0, b_i\}$ is given by $P=\left[
                                         \begin{array}{cc}
                                           1 & n \\
                                           1 & -1 \\
                                         \end{array}
                                       \right]$, and the structure constants are given by $b^2_1=nb_0+(n-1)b_1$. (Note that, for $n\in \ZZ^+$, $P$
is the first eigenmatrix of an association scheme of order $n+1$ and rank $2$.) Apply the two-step scaling on the matrix $P$, we obtain $S=\dfrac{1}{\sqrt{n+1}}\left[
                                         \begin{array}{cc}
                                           1 & \sqrt{n} \\
                                           \sqrt{n} & -1 \\
                                         \end{array}
                                       \right]$.
But for $1\neq n\in \RR^+$, the condition                                     ${\lambda_{ijk}\sqrt{\delta(b_k)}}/{\sqrt{\delta(b_i)\delta(b_j)}}\in \ZZ$ is not satisfied, because  $\lambda_{111}/\sqrt{n} \not\in \ZZ$,  and the matrix $S$ is not a unitary matrix.
\end{example}

Cuntz proved that the size of every integral Fourier matrix with odd rank is a square integer, \cite[Lemma 3.7]{MC1}. In the following lemma we generalize the Cuntz's result.

\begin{lemma}\label{RealMatrixSymmetricLemma}
Let $(A,\BB, \delta)$ be a $C$-algebra arising from a Fourier matrix $S$ with rank $r$. Let $r$ be an odd integer. If determinant of the $P$-matrix is an integer then the order of $A$ is a square integer.
\end{lemma}

\begin{proof}
Let $\mbox{det}(P)$ denote the determinant of $P$-matrix. $A$ is a self dual $C$-algebra, thus $P\bar P=nI$ implies $(\mbox{det}(P))^2=n^r$, where $n$ is the order of $A$.  Since $r$ is an odd integer, $\sqrt{n}$ is an integer.
\end{proof}

\section{Homogeneous $C$-algebras arising from $s$-matrices}

In this section we show that every homogenous $C$-algebra arising from a Fourier matrix has unique degree, and under certain condition 
it is a group algebra.
In the following lemma we prove that if a degree of a $C$-algebra arising from a Fourier matrix divides the all other nontrivial degrees then that degree might be equal to $1$.

\begin{lemma}\label{DividingAllenCaseThm}Let $(A,\BB, \delta)$ be a $C$-algebra  arising from a Fourier matrix $S$  such that the degrees $\delta(b_i)\in \ZZ$, for all $i$.  If for a given $j$, $\delta(b_j)$ divides $\delta(b_i)$ for all $i\geq 1$, then  $\delta(b_j)=1$, for all $j$.
\end{lemma}

\begin{proof} The degrees $\delta(b_i)$ are integers, therefore, by Proposition \ref{SymmetrizingPropertyProp}(ii), $d_j$ are integers.  Since $d_0=1+\sum\limits^{r-1}_{i=1}\delta(b_i)$, $ ({1+\sum\limits^{r-1}_{i=1}\delta(b_i)}){\delta(b_j)^{-1}} =\delta(b_j)^{-1}+\alpha\in \ZZ$, where $\alpha\in \ZZ$. Hence $\delta(b_j)=1$, for all $j$.
\end{proof}

Note that, the above lemma is true for any self-dual $C$-algebra with integral norms and degrees.
We can apply this lemma to recognize some $C$-algebras not arising from $s$-matrices just by looking at the degree pattern (or the first row of the  character table) of a $C$-algebra. For example, the character tables of the association schemes {\tt as5(2),  as9(2), as9(3), as9(8)} and {\tt as9(9)} violate the above result so they are not the character tables of the adjacency algebras arising from Fourier matrices. For the character tables of the association schemes see \cite{HM}.

\begin{definition}
Let $(A,\BB, \delta)$ be a $C$-algebra arising from a Fourier matrix $S$. 
Let $t\in \RR^+$ and $\delta(b_i)=t$,  for all $1\leq i\leq r-1$. Then $A$ is called a \emph{homogenous $C$-algebra} with \emph{homogeneity degree} $t$, and the associated Fourier matrix $S$ is called a \emph{homogeneous Fourier matrix}.
\end{definition}

In the next proposition, we prove that if a $C$-algebra arising from a Fourier matrix $S$ is either homogeneous,  or of prime order with integral degrees  then each degree of the algebra is equal to $1$.

\begin{prop}\label{HomogeneousAllenCaseAndPrimeOrderThm}Let $(A,\BB, \delta)$ be a $C$-algebra  arising from a Fourier matrix $S$.
\begin{enumerate}
  \item  If $A$ is a homogenous $C$-algebra then  $A$ has unique degree, that is, $\delta(b_i)=1$, for all $i$.
  \item  If the order of $A$ is a prime number and $\delta(b_i) \in \ZZ^+$,  for all $i$, then $\delta(b_i)=1$, for all $i$.
\end{enumerate}
\end{prop}

\begin{proof}(i). Let $t$ be the homogeneity degree of $A$. If the rank of $A$ is $2$ then the result holds trivially, see Example \ref{AllenIntConditionNotSatisfiedEx}.
Suppose $t\in \RR^+\backslash \QQ^+$.  Let $b_i, b_j$ be two nonidentity elements of $\BB$ such that $b_j\neq b_{i^*}$. Therefore, the support of $\tilde b_i\tilde b_j$ does not contain the identity element, because $\lambda_{ij0}=0$ implies $N_{ij0}=0$. Note that the first entry of the column vector  $\tilde b_i\tilde b_j$ is $t$. To obtain $t$ from the linear combinations of the nonidentity elements of $\tilde \BB$ the structure constants $N_{ijk}$ must involve $\sqrt{t}$, because the first entry of each column of $s$-matrix is $\sqrt{t}$ except the first column.
Therefore, the structure constants $N_{ijk}$ cannot be integers, a contradiction. Thus $t\in \QQ^+$. The entries of the $P$-matrix are algebraic integers, thus $t$ must be an integer. Therefore, $t$ divides $\delta(b_i)$ for all $i>0$. Hence by Lemma \ref{DividingAllenCaseThm}, $t=1$, that is, $\delta(b_i)=1$, for all $i$.

\smallskip

(ii). By Proposition \ref{SymmetrizingPropertyProp} (ii), $\delta(b_i)$ divides $d_0$,  for all $i$. 
But $d_0$ is a prime number, therefore, $\delta(b_i)=1$, for all $i$.
\end{proof}

Michael Cuntz made a conjecture that is a generalization to  his result, see  \cite[Lemma 3.11]{MC1}. The conjecture states that  if $s$ is  an integral Fourier matrix with unique norm, then $s \in \{\pm1\}^{r \times r}$, see \cite[Conjecture 3.10]{MC1}.
The conjecture cannot be extended to non-real $s$-matrices, the trivial contradiction is the character tables of cyclic groups of prime order.
Note that if all the structure constants are nonnegative then $|p_{ij}|\leq p_{0j}$ (or equivalently, $|s_{ij}|\leq s_{0j}$) for all $j$, see \cite[Proposition 4.1]{Xu1}. The next theorem completely classify the $s$-matrices under the conditions of Proposition \ref{HomogeneousAllenCaseAndPrimeOrderThm} and $|p_{ij}|\leq p_{0j}$, for all $j$.
It also shows that even for real $s$-matrices it is not necessary to pre-assume the uniqueness of the norm to get all entries of $s$-matrices equal to $\pm 1$.

\begin{thm}\label{ClassificationOfAllenMatricesThm} Let $(A,\BB, \delta)$ be a $C$-algebra arising from a Fourier matrix $S$. Suppose $A$ is either a homogeneous $C$-algebra, or  $A$ has prime order and integral degrees. Let $|s_{ij}|\leq s_{0j}$, for all $i,j$.
\begin{enumerate}
 \item The modulus of each entry of the $s$-matrix is $1$, that is,  $|s_{ij}|=1$, for all $i,j$.
 \item The columns of the $s$-matrix form an abelian group under entrywise multiplication.
\item If $s$-matrix is a real matrix then the columns of $s$-matrix form an elementary abelian group.
\end{enumerate}
\end{thm}

\begin{proof}(i). By Theorem \ref{HomogeneousAllenCaseAndPrimeOrderThm}, $\delta(b_j) =1$, for all $j$. Thus $d_j=d_0$, for all $j$. Therefore, $|s_{ij}|\leq 1$, hence $|s_{ij}|=1$, for all $i,j$.

\smallskip

(ii). Let $s_i$ be the $i$-th column of $s$-matrix and let $s_is_j$ be the entrywise multiplication of $s_i$ and $s_j$.
For a column $s_k$ of $s$, each entry of $s_i s_j s_k$ has modulus value $1$, and the sum of the entries of $s_i s_j s_k$  cannot be $-r$ because $s_{0l}=1$, for all $l$. But the structure constants $N_{ijk}=\dfrac{1}{r}\sum_l s_{li}s_{lj}s_{lk}$ are integers. Therefore, $N_{ijk}$ are either $0$ or $r$. Since  $s_is_j$ is a nonzero vector, it cannot be orthogonal to all the columns of $s$-matrix. Thus $s_is_j=s_k$ for some column $s_k$ of $s$. Therefore, the columns of $s$-matrix are closed under entrywise multiplication. The first column of $s$-matrix serves as an identity of the group of columns of $s$-matrix under entrywise multiplication.

\smallskip

(iii). Let $s$ be a real matrix. Since $|s_{ij}|=1$, $s_{ij}=\pm 1$. Thus, by Part $(ii)$,  $s$ is a group. Hence $s$ is a character table of an elementary abelian group of order $r$.
\end{proof}

If an $s$-matrix  of rank $r$ has unique norm then the Fourier matrix $S= r^{-1/2}s$. Thus the above theorem also classifies the homogeneous Fourier matrices under the same condition.


\begin{thebibliography}{10}

\bibitem{AFM} Z.~Arad, E.~Fisman, and M.~Muzychuk, Generalized table algebras, {\it Israel J. Math.}, {\bf 114} (1999), 29-60.


\bibitem{BI} E. Bannai and T. Ito, {\it Algebraic Combinatorics I: Association Schemes}, Benjamin/Cummings, Menlo Park, CA, 1984.

\bibitem{HIB1} Harvey I. Blau, Quotient structures in $C$-algebras, {\it J. Algebra}, {\bf 175} (1995), no.~1, 24-64; Erratum: {\bf 177} (1995), no.~1, 297-337.

\bibitem{HIB2} Harvey I. Blau, Table algebras, {\it European J. Combin.}, {\bf 30} (2009), no.~6, 1426-1455.


\bibitem{MC1} Michael Cuntz, Integral modular data and congruences,\emph{ J Algebr Comb} {\bf29}, (2009), 357-387.


\bibitem{HM} A.~Hanaki and I.~Miyamoto, Classification of Small Association Schemes. (http://math.shinshu-u.ac.jp/~hanaki/as/)

\bibitem{HS2} Allen Herman and Gurmail Singh, Torsion units of integral C-algebras, {\it JP Journal of Algebra, Number Theory, and Applications}, {\bf 36} (2), 2015, 141-155.

\bibitem{Hig87} D. G. Higman, Coherent algebras, {\it Linear Algebra Appl.}, {\bf 93} (1987), 209-239.

\bibitem{Xu1} Bangteng Xu, Characters of table algebras and applications to
association schemes, {\it Journal of Combinatorial Theory}, Series A {\bf 115} (2008) 1358-1373.
\end{thebibliography}
\end{document}